\documentclass{amsproc}

\usepackage{amsmath} 
\usepackage{amsthm} 
\usepackage[pdftex]{graphicx} 
\usepackage{hyperref} 
\usepackage{amsfonts} 

\def\C{\mathbb {C}} 
\def\Q{\mathbb {Q}} 
 
\def\N{\mathbb {N}} 
\def\Z{\mathbb {Z}}
\def\ZK{\Z_K}
\def\ZKtimes{\ZK^\times}

\def\thetanu{\nu}

\def\atop#1#2{
\genfrac{}{}{0pt} {} 
{#1} 
{#2}}

\def\epsilonbar{\overline{\epsilon}}

\def\house#1{\setbox1=\hbox{$\,#1\,$}%
\dimen1=\ht1 \advance\dimen1 by 2pt \dimen2=\dp1 \advance\dimen2 by 2pt
\setbox1=\hbox{\vrule height\dimen1 depth\dimen2\box1\vrule}%
\setbox1=\vbox{\hrule\box1}%
\advance\dimen1 by .4pt \ht1=\dimen1
\advance\dimen2 by .4pt \dp1=\dimen2 \box1\relax}

\def\rmN{\mathrm{N}}
\def\trace{\mathrm{Tr}}
\def\rmNorm{\mathrm{N}}

\def\rmh{\mathrm{h}}
\def\tors{{\mathrm{tors}}}
\def\sff{\mathsf{f}}
\def\sfA{\mathsf{A}}
\def\sfF{\mathsf{F}}

\newtheorem{theorem}{Theorem}[section]
\newtheorem{theoreme}[theorem]{Th\'eor\`eme} 
\newtheorem{conjecture}[theorem]{Conjecture} 
\newtheorem{proposition}[theorem]{Proposition}
\newtheorem{lemme}[theorem]{Lemme}
\newtheorem{corollaire}[theorem]{Corollaire}

\theoremstyle{definition}

\theoremstyle{remark}

\numberwithin{equation}{section}

\def\og{\leavevmode\raise.3ex\hbox{$\scriptscriptstyle
\langle\!\langle$}}

\def\fg{\leavevmode\raise.3ex\hbox{$\scriptscriptstyle
\,\rangle\!\rangle$}}

\newcounter{compteurkappa} 

\def\Newcst#1{
\refstepcounter{compteurkappa}
\kappa_{ 
\arabic{compteurkappa}}
\label{#1}
}

\def\cst#1{\kappa_{\ref{#1}}}
\def\boxit#1#2{\setbox1=\hbox{\kern#1{#2}\kern#1}%
\dimen1=\ht1 \advance\dimen1 by #1 \dimen2=\dp1 \advance\dimen2 by #1
\setbox1=\hbox{\vrule height\dimen1 depth\dimen2\box1\vrule}%
\setbox1=\vbox{\hrule\box1\hrule}%
\advance\dimen1 by .4pt \ht1=\dimen1
\advance\dimen2 by .4pt \dp1=\dimen2 \box1\relax}

\begin{document}
 
 \hfill 
 
 \null
 

 \bigskip

 \title[Familles d'\'equations de Thue]{Familles d'\'equations de Thue associ\'ees \`a 
un sous-groupe de rang $\mathbf{1}$ d'unit\'es
totalement r\'eelles d'un corps de nombres}


\author{Claude Levesque}
\address{D\'epartement de math\'ematiques et de statistique
\\
Universit\'e Laval, Qu\'ebec (Qu\'ebec)
\\
CANADA G1V 0A6}
\curraddr{}
\email{Claude.Levesque@mat.ulaval.ca}
\thanks{}

\author{Michel Waldschmidt}
\address{UPMC Univ Paris 06, UMR 7586-IMJ
\\
75005 Paris
\\
FRANCE}
\curraddr{}
\email{michel.waldschmidt@imj-prg.fr}
\thanks{}

\dedicatory{D\'edi\'e \`a {\sc Ram Murty} \`a l'occasion de son $60^e$ anniversaire de naissance}.

\subjclass[2010]{Primary 
11D61  Secondary 11D25 11D41 11D59}


\begin{abstract} 
Let $F$ be an irreducible binary form attached to a number field $K$ of degree $\geq 3$. Let $\epsilon\not\in \{-1, 1\}$ be a totally real unit of $K$. By twisting $F$ with the powers $\epsilon^a$ of $\epsilon$, ($a\in\Z$), we obtain an infinite family $F_a$ of binary forms. Let $m\in\Z$. We give an effective bound for 
$\max\{|a|, \log|x|, \log|y|\}$ when $a,x,y$ are rational integers satisfying $F_a(x,y)=m$ with $xy\not=0$.
\\
\bigskip
\bigskip

\noindent
{\sc R\'esum\'e.}
Soit $F$ une forme binaire irr\'eductible attach\'ee \`a un corps de nombres $K$ de degr\'e $\ge 3$. Soit $\epsilon\not \in \{-1,1\}$ une unit\'e totalement r\'eelle de $K$. En tordant $F$ par les puissances $\epsilon^a$ de $\epsilon$, ($a\in\Z$), nous obtenons une famille infinie $F_a$ de formes binaires. Soit $m\in\Z$. Nous donnons une borne effective pour $\max\{|a|, \log|x|, \log|y|\}$ quand $a,x,y$ sont des entiers rationnels satisfaisant $F_a(x,y)=m$ avec $xy\not=0$.
\\
\end{abstract}

\maketitle

  {\bf Mots clefs:} 
\'equations de Thue, formes binaires, \'equations diophantiennes, bornes effectives.

\section{Le r\'esultat principal }\label{S:ResultatPrincipal}

 Soit $\alpha$ un nombre alg\'ebrique de degr\'e $d\ge 3$ sur $\Q$. On d\'esigne par $K$ le corps de nombres $\Q(\alpha)$, par $f\in \Z[X]$ le polyn\^ome irr\'eductible de $\alpha$ sur $\Z$ et par $ \ZKtimes$ le groupe des unit\'es de $K$. \`a chaque unit\'e $\varepsilon\in \ZKtimes$ dont le degr\'e $r=[\Q(\alpha\varepsilon):\Q]$ est $\geq 3$, on attache le polyn\^ome irr\'eductible $\sff_{\varepsilon}(X)\in \Z[X]$ de $\alpha\varepsilon$ sur $\Z$ (d\'efini de fa\c{c}on unique quand on impose que le coefficient directeur soit $>0$) et par $\sfF_{\varepsilon}$ la forme binaire irr\'eductible 
$$
\sfF_{\varepsilon}(X,Y) = Y^r \sff_{\varepsilon}(X/Y)\in\Z[X,Y].
$$
 On d\'esigne par $\rmh$ la hauteur logarithmique absolue. 
Rappelons la conjecture $1$ de \cite{LW2}. 
 
\begin{conjecture}
\label{ConjecturePrincipale}
Il existe une constante effectivement calculable $\Newcst{kappaConjecture}>0$, ne d\'ependant que de $\alpha$, telle que, pour tout $m\ge 2$, toute solution $(x,y,\varepsilon)\in\Z\times \Z\times\ZKtimes$ de l'in\'egalit\'e 
$$
| \sfF_\varepsilon(x,y)|\le m, \mbox{ avec }\; xy\neq 0\;\mbox{ et } \; [\Q(\alpha\varepsilon) : \Q]\ge 3,
$$
 satisfait
$$
 \max\{|x|,\; |y|,\; e^{\rmh(\alpha\varepsilon)}\}\le m^{\cst{kappaConjecture}}.
$$
\end{conjecture}

Au cours de cet article, nous nous proposons de prouver le cas particulier de cette conjecture o\`u on restreint les solutions $(x,y,\varepsilon)$ \`a un ensemble \mbox{$\Z\times \Z\times U$}, o\`u $U$ est un sous-groupe de rang $1$ de $\ZKtimes$ engendr\'e par une unit\'e totalement r\'eelle.
 
Soit $\alpha$ un nombre alg\'ebrique de degr\'e $d\ge 3$. Soit $\epsilon$ une unit\'e d'ordre infini du corps de nombres $\Q(\alpha)$. 
Quand $a$ est un nombre entier tel que $\Q(\alpha\epsilon^a)=\Q(\alpha)$, on d\'esigne par $F_a\in \Z[X,Y]$ la forme binaire irr\'eductible de degr\'e $d$ telle que $F_a( \alpha\epsilon^a,1)=0$ et $F_a(1,0)>0$. Avec les notations de la conjecture $\ref{ConjecturePrincipale}$ on a 
$F_a=\sfF_{\epsilon^a}$.
Ainsi, en d\'esignant par $\Phi$ l'ensemble des plongements de $K$ dans $\C$, le polyn\^ome irr\'eductible de $\alpha$ sur $\Z$ est
$$
F_0(X,1)=a_0\prod_{\varphi\in\Phi}\bigl(X-\varphi(\alpha)\bigr)
$$
avec $a_0=F_0(0,1)$, tandis que le polyn\^ome irr\'eductible de $\alpha\epsilon^a$ sur $\Z$ est
$$
F_a(X,1)=a_0\prod_{\varphi\in\Phi}\bigl(X-\varphi(\alpha\epsilon^a)\bigr).
$$ 
Le th\'eor\`eme que nous d\'emontrons est le suivant.

\begin{theoreme}\label{Theoreme:rang1}
Soit $\alpha$ un nombre alg\'ebrique de degr\'e $d\ge 3$. Soit $\epsilon$ une unit\'e totalement r\'eelle du corps $\Q(\alpha)$. 
Il existe une constante effectivement calculable $\Newcst{kappaTheoreme}>0$, ne d\'ependant que de $\alpha$ et $\epsilon$, telle que, pour tout $m\ge 2$, tout triplet $(x,y,a)\in\Z^3$ satisfaisant 
$$
| F_a(x,y)|\le m,\; \mbox{ avec } \; xy\neq 0\;\mbox{ et } \; \Q(\alpha\epsilon^a) =\Q(\alpha),
$$
 v\'erifie
$$
 \max\{\log |x|,\; \log |y|,\;|a| \}\le \cst{kappaTheoreme}\log m. 
$$
\end{theoreme}
\medskip 

L'\'enonc\'e suivant a \'et\'e utilis\'e dans \cite{LW3}. 

\begin{corollaire}\label{Corollaire:cubiquetotalementreel}
Supposons le corps $K=\Q(\alpha)$ cubique. Soit $\epsilon$ une unit\'e de $K$. 
Il existe une constante effectivement calculable $\Newcst{kappacorollaire}>0$, ne d\'ependant que de $\alpha$ et $\epsilon$, telle que, pour tout $m\ge 2$, tout triplet $(x,y,a)\in\Z^3$ satisfaisant 
$$
| F_a(x,y)|\le m,\; \mbox{ avec } \; xy\neq 0 \; \mbox{ et }\; \Q(\alpha\epsilon^a) =\Q(\alpha),
$$ 
v\'erifie
$$
 \max\{\log |x|,\; \log |y|,\;|a| \}\le \cst{kappacorollaire}\log m. 
$$
\end{corollaire}

Ce corollaire se d\'eduit du th\'eor\`eme \ref{Theoreme:rang1} dans le cas o\`u le corps cubique $\Q(\alpha)$ est totalement r\'eel. Dans le cas contraire, le rang du groupe des unit\'es de $\Q(\alpha)$ est $1$, ce corps cubique n'admet qu'un plongement r\'eel, et le corollaire $\ref{Corollaire:cubiquetotalementreel}$ se d\'eduit alors des r\'esultats de \cite{LW-Balu}. 

La section $\ref{Section:LemmeElementaire}$ est consacr\'ee \`a un lemme \'el\'ementaire qui sera utilis\'e plusieurs fois dans la d\'emonstration. La d\'emonstration du th\'eor\`eme \ref{Theoreme:rang1} se trouve au \S $\ref{Section:demonstration}$. Au \S 4, nous donnerons des familles d'exemples. 

L'outil principal de notre texte est une in\'egalit\'e diophantienne \'enonc\'ee au lemme $\ref{VarianteLemme3}$, et la d\'emonstration ressemble \`a celle 
du lemme 3 de \cite{LW2}.

Le pr\'esent texte repose sur les r\'esultats et sur les d\'emonstrations de \cite{LW2}. Nous utilisons les notations de cet article, en pr\'ecisant quand nous devons les modifier. 
 
\section{Un lemme \'el\'ementaire}\label{Section:LemmeElementaire}

On utilisera le lemme suivant avec $t=4$, $5$ ou $6$. 

\begin{lemme}\label{Lemme:elementaire}
Soient $t$ un entier $\ge 3$, $x_1,\dots,x_t$ des nombres r\'eels, $\delta$ et $\mu$ deux nombres r\'eels positifs satisfaisant 
$$
0<\delta\le \frac{1}{t-2}-\frac{1}{\mu}\cdotp
$$
On suppose $x_1+\cdots+x_t=0$ et 
$$
|x_i|\le \delta \max\{|x_1|, |x_2|\}
\quad\hbox{pour}\quad i=3,\dots,t.
$$
Alors
$$
|x_1+x_2|
\le \mu\delta \min\{|x_1|, |x_2|\}.
$$
\end{lemme}

\begin{proof}[D\'emonstration]
Par sym\'etrie, on peut supposer $|x_1|\le |x_2|$. Si $x_1=0$, alors les hypoth\`eses impliquent $x_2=0$ et le lemme est vrai. Supposons donc $x_1\not=0$. 
Posons
$$
s=\frac{x_3+\cdots+x_t}{x_2}\cdotp
$$
On a $x_1=-x_2(1+s)$, d'o\`u on d\'eduit 
$$
\frac{x_2}{x_1}+1=\frac{s}{1+s}\cdotp
$$
Comme 
$$
|s|\le (t-2)\delta<1,
$$
on peut \'ecrire
$$
\left|\frac{s}{1+s}\right|\le \frac{|s|}{1-|s|} \le \frac{(t-2)\delta}{1-(t-2)\delta}\le 
\mu\delta.
$$ 
\end{proof}
 
Nous utiliserons ce lemme \ref{Lemme:elementaire} avec $\mu=t$, de sorte que l'hypoth\`ese sur $\delta$ devient
$$
0<\delta\le \frac{2}{t(t-2)}\cdotp
$$
D'autres choix sont licites mais ne modifient pas le r\'esultat, dans la mesure o\`u nous ne cherchons pas \`a expliciter les constantes. 

\section{D\'emonstration du th\'eor\`eme \ref{Theoreme:rang1}}\label{Section:demonstration}

Dans cette section, on suppose que le corps $K=\Q(\alpha)$ est totalement r\'eel mais on ne se restreint pas au cas cubique. 
On utilise les notations et r\'esultats de \cite{LW2}, \`a ceci pr\`es que, par rapport \`a \cite{LW2}, nous rempla\c{c}ons les notations $\tau_a$ par $\tau_\alpha$, $\sigma_a$ par $\sigma_\alpha$, $T_a(\thetanu)$ par $T_\alpha(\thetanu)$, $\Sigma_a(\thetanu)$ par $\Sigma_\alpha(\thetanu)$. Il n'y aurait pas de conflit de notations \`a conserver les notations $\tau_b$, $\sigma_b$, $T_b(\thetanu)$ et $\Sigma_b(\thetanu)$, mais par souci de sym\'etrie nous rempla\c{c}ons les indices $b$ par des $\beta$. 

Ainsi nous notons $\sigma_\alpha$ (resp{.} $\sigma_\beta$) un plongement de  $K$ dans $\C$ tel que $|\sigma_\alpha(\alpha\varepsilon)|$ (resp{.} $|\sigma_\beta(\beta)|$) soit maximal parmi les \'el\'ements   $|\varphi(\alpha\varepsilon)|$ (resp{.} parmi les \'el\'ements  $|\varphi(\beta)|$) pour $\varphi\in\Phi$. Donc
$$
|\sigma_\alpha(\alpha\varepsilon)|=\house{\alpha\varepsilon}
\quad
\hbox{et}
\quad
|\sigma_\beta(\beta)|=\house{\beta}\; .
$$ 
Ensuite nous d\'esignons par $\tau_\alpha$ (resp{.} $\tau_\beta$) un plongement de  $K$ dans $\C$ tel que  $|\tau_\alpha(\alpha\varepsilon)|$ (resp{.} $|\tau_\beta(\beta)|$) soit minimal
parmi les \'el\'ements $|\varphi(\alpha\varepsilon)|$ (resp{.} parmi les \'el\'ements  $|\varphi(\beta)|$) pour $\varphi\in\Phi$.
Donc\footnote{Ceci corrige deux fautes de frappe \`a la page 128 de \cite{LW2}}
$$
\left|\tau_\alpha\left((\alpha\varepsilon)^{-1}\right)\right|=
\house{ {(\alpha\varepsilon)^{-1}}}
\quad
\hbox{et}
\quad
\left|\tau_\beta\left(\beta^{-1}\right)\right|=
\house{
\beta^{-1}} \ .
$$ 
Soit $\thetanu$ un nombre r\'eel dans l'intervalle ouvert 
$]0,1[$. 
Nous d\'esignons par  $\Sigma_\alpha(\thetanu)$, \, $\Sigma_\beta(\thetanu), $\, $T_\alpha(\thetanu), $ \, $ T_\beta(\thetanu)$\,\, 
les ensembles de plongements de  $K$ dans $\C$ satisfaisant les conditions 
$$
\left\{
\begin{array}{lll}
\Sigma_\alpha(\thetanu)&=&\left\{
\varphi\in\Phi\; 
\mid
|\sigma_\alpha(\alpha\varepsilon)|^\thetanu \le |\varphi(\alpha\varepsilon)|\le |\sigma_\alpha(\alpha\varepsilon)|
\right\},
\\ [2mm]
\Sigma_\beta(\thetanu)&=&\left\{
\varphi\in\Phi\; 
\mid
|\sigma_\beta(\beta)|^\thetanu \le |\varphi(\beta)|\le |\sigma_\beta(\beta)|
\right\},
 \\ [2mm]
T_\alpha(\thetanu)&=&\left\{
\varphi\in\Phi\; 
\mid
|\tau_\alpha(\alpha\varepsilon)|\le |\varphi(\alpha\varepsilon)|\le |\tau_\alpha(\alpha\varepsilon)|^\thetanu 
\right\},
\\ [2mm]
T_\beta(\thetanu)&=&\left\{
\varphi\in\Phi\; 
\mid
|\tau_\beta(\beta)|\le |\varphi(\beta)|\le |\tau_\beta(\beta)|^\thetanu 
\right\}.
\end{array} \right.
$$
On se place sous les hypoth\`eses de la conjecture $\ref{ConjecturePrincipale}$, mais en se restreignant \`a un sous-groupe de rang $1$ du groupe des unit\'es de $K=\Q(\alpha)$. On fixe donc une unit\'e $\epsilon$ d'ordre infini de $K$ (ce qui implique $[K:\Q]\ge [\Q(\alpha\epsilon):\Q] \ge 3$) et l'\'el\'ement $\varepsilon$ de \cite{LW2} devient ici $\epsilon^a$. On suppose de plus que nous sommes dans une situation o\`u l'unit\'e $\epsilon$ est totalement r\'eelle. Les constantes $\kappa_i$ qui suivent ne d\'ependent que de $\alpha$ et $\epsilon$. On utilisera le lemme $\ref{VarianteLemme3}$ avec diff\'erentes valeurs de $\lambda$ qui pourront d\'ependre de $\alpha$ et $\epsilon$, mais qui seront ind\'ependantes de $a$, $x$, $y$ et $m$. 
 
\medskip
\indent
{\tt Premi\`ere \'etape.}
On consid\`ere un quadruplet $(x,y,a,m)$ d'entiers v\'erifiant 
$$
|F_a(x,y)|\le m \; \mbox{ avec } \; 
xy\not=0,\quad
\Q(\alpha\epsilon^a)=K,\quad
\quad m\ge 2.
$$
Quitte \`a remplacer $\alpha$ par $\alpha^{-1}$, $\epsilon$ par $\epsilon^{-1}$ et \`a permuter $x$ et $y$, on peut supposer
$$
1\le |x|\le |y|
\quad \hbox{et}\quad 
a\ge 0.
$$
Comme dans \cite{LW2}, on d\'esigne par $\epsilon_1,\dots,\epsilon_r$ des unit\'es de $K$ dont les classes mo\-du\-lo $K^\times_{\tors}$ forment une base du groupe ab\'elien libre $\ZKtimes/K^\times_{\tors}$.
Au \S 3 de \cite{LW2}, on a introduit deux param\`etres $A$ et $B$ reli\'es \`a la hauteur logarithmique absolue de $\alpha\epsilon^a$ et de $\beta=x-\alpha\epsilon^ay$ respectivement. Pour cela on a \'ecrit 
$$ 
\varepsilon=\zeta \epsilon_1^{a_1}\cdots\epsilon_r^{a_r},
\quad
\beta=\rho \epsilon_1^{b_1}\cdots\epsilon_r^{b_r}
$$ 
avec des entiers rationnels $a_1,\dots,a_r$, $b_1,\ldots,b_r$, avec $\zeta\in K^\times_{\tors}$ et avec $\rho \in\ZK$ 
v\'erifiant $\rmh(\rho )\le \Newcst{ell} \log m$. Alors
$$
A=\max\{1,|a_1|, \dots,|a_r|\}
\quad
\hbox{et}
\quad
B= \max\bigl\{1,\; |b_1|, |b_2|,\ldots, \; |b_r|\bigr\}.
$$
Ici, on peut remplacer $A$ par $\Newcst{kappa:A} a$. 
Les lemmes 4 et 5 de \cite{LW2} donnent 
$$
\Newcst{kappa:majA} A\le B \le \Newcst{kappa:minA} A
\quad
\hbox{et}\quad
|y|\le e^{\Newcst{kappa:majy} B}.
$$
Nous allons voir que si 
 $ \Newcst{kappa:LW2a}$ d\'esigne une constante suffisamment grande,
 les minorations 
\begin{equation}\label{Equation:MinorationAetB}
A\ge \cst{kappa:LW2a}\log m\quad \hbox{et}\quad B\ge \cst{kappa:LW2a}\log m
\end{equation}
(cf. \'equation (8) de \cite{LW2}) entra\^{\i}nent une contradiction. Cette contradiction terminera donc la d\'emonstration. 

Les \'equations (9) de \cite{LW2} donnent ici qu'il existe des constantes positives effectivement calculables $\Newcst{kappa:majsigma}$ et $\Newcst{kappa:minsigma}$, ne d\'ependant que de $\alpha$ et $\epsilon$, telles que 
 \begin{equation}
 \label{Equation:MajMinsigma}
\left\{
\begin{array}{lllll}
e^{\cst{kappa:minsigma} A}
&\leq& |\sigma_\alpha(\alpha\epsilon^a)| &\leq& e^{\cst{kappa:majsigma} A}, \\ [1mm]
e^{\cst{kappa:minsigma} B}
&\leq& |\sigma_\beta(\beta)| &\leq& e^{\cst{kappa:majsigma} B}, \\ [1mm]
e^{-\cst{kappa:majsigma} A} 
 &\leq& |\tau_\alpha(\alpha\epsilon^a)|&\leq& e^{- \cst{kappa:minsigma} A},\\ [1mm]
e^{-\cst{kappa:majsigma} B} 
 &\leq& |\tau_\beta(\beta)|&\leq& e^{- \cst{kappa:minsigma} B}.\\
\end{array}\right.
\end{equation} 
 Comme $\alpha$ est diff\'erent de $\pm1$, un de ses conjugu\'es est en valeur absolue $>1$, un autre est en valeur absolue $<1$; donc 
\begin{equation}
\label{Equation:MajMintaualphasigmaalpha}
|\tau_\alpha(\epsilon)|<1<|\sigma_\alpha(\epsilon)|.
\end{equation}
Les in\'egalit\'es dans ($\ref{Equation:MajMinsigma}$) concernant $\tau_\alpha $ et $\sigma_\alpha$ sont \'equivalentes \`a  celles de ($\ref{Equation:MajMintaualphasigmaalpha}$).
 
\medskip
 
\indent
{\tt Deuxi\`eme \'etape.} 
L'\'enonc\'e qui suit est une variante du lemme 3 de \cite{LW2}, dans lequel $\lambda=-1 $. 

\begin{lemme}\label{VarianteLemme3}
Soit $\lambda$ un \'el\'ement non nul de $\Q(\alpha)$. 
Il existe une constante po\-si\-tive effectivement calculable $\Newcst{kappa:lemme}$, d\'ependant non seulement de $\alpha$ et $\epsilon$ mais aussi de $\lambda$, ayant la propri\'et\'e suivante. 
Soient $\varphi_1,\varphi_2,\varphi_3,\varphi_4$ des \'el\'ements de $\Phi$ tels que $\varphi_1(\alpha\epsilon^a)\varphi_2(\beta) \not=-\lambda \varphi_3(\alpha\epsilon^a)\varphi_4(\beta)$. Alors 
$$
\left|
 \frac{\varphi_1(\alpha\epsilon^a)\varphi_2(\beta) }
 {\varphi_3(\alpha\epsilon^a)\varphi_4(\beta)} 
 +\lambda\right|\ge 
 \exp\left(
- \cst{kappa:lemme} 
 (\log m)\log
 \left(2+
 \frac{A+B}{\log m}\right)
 \right). 
$$
 \end{lemme}

\begin{proof}[D\'emonstration]
La d\'emonstration ressemble \`a celle du lemme 3 de \cite{LW2}.  On \'ecrit 
$$
- \frac{1}{\lambda}\cdot 
 \frac{\varphi_1(\alpha\epsilon^a)\varphi_2(\beta) }
 { \varphi_3(\alpha\epsilon^a)\varphi_4(\beta)} 
$$
 sous la forme $\gamma_1^{c_1}\cdots \gamma_s^{c_s}$ avec $s=r+2$, et
$$ \begin{array}{rcl}
 \gamma_{j}=\displaystyle \frac{ \varphi_2(\epsilon_j)}{\varphi_4(\epsilon_j)},&
c_{j}=b_j
 \quad (j=1,\dots, r), &
 \gamma_{r+1}=\displaystyle \frac{ \varphi_1(\epsilon)}{\varphi_3(\epsilon)},\\ [3mm]
c_{r+1}=a, &\gamma_s=\displaystyle - \frac{1}{\lambda}\cdot \frac{\varphi_1(\alpha\zeta)\varphi_2(\rho)} {\varphi_3(\alpha\zeta)\varphi_4(\rho)},&
c_s=1.
\end{array}
$$ 
On a $\rmh(\gamma_s)\le \Newcst{hauteurgammas} \log m$.
On utilise enfin la proposition 2 de \cite{LW2} avec 
$$ 
 H_1=\cdots=H_{r+1}=\Newcst{HiLemmeStrategie}, \quad H_s=\cst{HiLemmeStrategie} \log m, \quad C=2+
 \frac{A+B }{\log m}\cdotp
$$
 Le lemme $\ref{VarianteLemme3}$ en r\'esulte.
\end{proof}
 
\medskip

\goodbreak 

\indent
{\tt Troisi\`eme \'etape.} 
Montrons que l'on a $\tau_\alpha \not= \tau_\beta $.

Supposons $\tau_\alpha = \tau_\beta $.
On d\'esigne par $\varphi$ un \'el\'ement de $\Phi$ distinct de $\sigma_\alpha$ tel que $ \varphi(\epsilon)\not=\pm \tau_\alpha(\epsilon)$. L'existence de $\varphi$ est claire si 
$[\Q(\epsilon):\Q]\ge 4$; elle est vraie aussi si $\Q(\epsilon)$ est un corps cubique: dans ce cas il suffit de prendre $\varphi\not=\sigma_\alpha$ et $\varphi\not=\tau_\alpha$ car la somme de deux conjugu\'es d'un nombre alg\'ebrique de degr\'e impair n'est jamais nulle\footnote{Si un nombre alg\'ebrique non nul $\gamma$ est conjugu\'e de $-\gamma$, le polyn\^ome irr\'eductible $P$ de $\gamma$ sur $\Q$ s'annule au point $-\gamma$, donc $P(X)=\pm P(-X)$. Comme $P(0)\not=0$, on a $P(X)=P(-X)$, par cons\'equent il existe un polyn\^ome $Q$ tel que $P(X)=Q(X^2)$. En particulier, le degr\'e de $P$ est pair.}. 

\indent D'apr\`es le lemme 6 de \cite{LW2}, on peut supposer $\varphi\not\in T_\beta(\thetanu)$. D'apr\`es le corollaire 1 de \cite{LW2}, on peut supposer $\varphi\not\in \Sigma_\alpha(\thetanu)$. 
On \'ecrit maintenant l'\'equation ($7$) de \cite{LW2}
\begin{equation}\label{Equation:SommeSixTermes}
u_1v_2-u_1v_3+u_2v_3-u_2v_1+u_3v_1-u_3v_2=0
\end{equation} 
 avec
$$\left\{ \begin{array}{rcl}
u_1=\varphi(\alpha\epsilon^a),&
u_2=\sigma_\alpha(\alpha\epsilon^a), &
u_3=\tau_\alpha(\alpha\epsilon^a),\\ [2mm] 
v_1=\varphi(\beta), &
v_2=\sigma_\alpha(\beta), &
v_3=\tau_\alpha(\beta).
\end{array} \right.
$$

On utilise les in\'egalit\'es de ($\ref{Equation:MajMinsigma}$) et ($\ref{Equation:MajMintaualphasigmaalpha}$). On en d\'eduit d'abord
$$
\frac{|u_3|}{|u_2|}=
 \frac{|\tau_\alpha(\alpha\epsilon^a)|}{|\sigma_\alpha(\alpha\epsilon^a)|}\le e^{-\Newcst{kappa:LW2b} A}.
$$
Comme $ \varphi(\epsilon)\not=\pm \tau_\alpha(\epsilon)$ et que l'unit\'e $\epsilon$ est totalement r\'eelle, on a
$$ 
|\varphi(\epsilon)|>|\tau_\alpha(\epsilon)|,
$$
d'o\`u
$$
\frac{|u_3|}{|u_1|}=
\frac{|\tau_\alpha(\alpha\epsilon^a) |}
 {|\varphi(\alpha\epsilon^a)|} =
\frac{|\tau_\alpha(\alpha ) |}
 {|\varphi(\alpha )|} 
 \cdot
 \left(
\frac{|\tau_\alpha(\epsilon ) |}
 {|\varphi( \epsilon )|} \right)^a
\le e^{- \Newcst{kappa:u3/u1a} A}.
$$
Comme $\tau_\alpha = \tau_\beta $ et que $\varphi\not\in T_\beta(\thetanu)$ et $\sigma_\alpha\not\in T_\beta(\thetanu)$, on a encore
$$
 \frac{|v_3|}{|v_1|}=
 \frac{|\tau_\beta(\beta)|}{|\varphi(\beta)|}\le e^{-\Newcst{kappa:v3/v1} B}
 \quad\hbox{et}\quad 
 \frac{|v_3|}{|v_2|}=
 \frac{|\tau_\beta(\beta) |}
 {|\sigma_\alpha(\beta)|} \le e^{-\Newcst{kappa:v3/v2} B}.
 $$
On veut utiliser le lemme \ref{Lemme:elementaire} avec $t=\mu=6$, $\delta\leq \frac{1}{12}\cdotp$ 
On a $6 $ termes $\pm u_i v_j \quad ( i\neq j)$ de somme nulle. On prend pour $x_1$ et $x_2$ les deux termes $u_1v_2$ et $-u_2v_1$,
respectivement. Il nous faut donc nous assurer que 
$$
\frac{|x_i|}{\max \{ |x_1|,|x_2|\}} \leq \delta,\qquad (i=3,4,5,6).
$$
En utilisant les majorations que nous venons d'\'etablir pour les modules de 
$$
\frac{u_3v_1}{u_2v_1}=\frac{u_3}{u_2},\quad
\frac{u_3v_2}{u_1v_2}=\frac{u_3}{u_1},\quad
\frac{u_1v_3}{u_1v_2}=\frac{v_3}{v_2},\quad
\frac{u_2v_3}{u_2v_1}=\frac{v_3}{v_1},
$$ 
le lemme \ref{Lemme:elementaire} avec $t=6$ nous donne alors une borne sup\'erieure pour $|x_1+x_2|$, \`a savoir
$$
|x_1+x_2| \leq 6\delta \min \{|x_1|,|x_2|\} ,
$$
de sorte que
$$
\left|\frac{x_1}{x_2}+1\right| \leq 6 \delta.
$$

\noindent Gr\^ace \`a l'\'egalit\'e
$$
\frac{u_1v_2}{u_2v_1}=
 \frac{\varphi(\alpha\epsilon^a)\sigma_\alpha(\beta) }
 {\sigma_\alpha(\alpha\epsilon^a)\varphi(\beta)},
 $$
 on d\'eduit du lemme \ref{Lemme:elementaire} que 
 $$
\left| 
 \frac{\varphi(\alpha\epsilon^a)\sigma_\alpha(\beta) }
 {\sigma_\alpha(\alpha\epsilon^a)\varphi(\beta)} -1\right|\le e^{-\Newcst{kappa:u1v2/u2v1} \min\{A,B\}}.
$$
 Utilisons maintenant le lemme $\ref{VarianteLemme3}$ avec $\lambda=-1$ et
 avec $\varphi_1=\varphi_4=\varphi$, $\varphi_2=\varphi_3=\sigma_\alpha$. L'hypoth\`ese 
$$
\varphi(\alpha\epsilon^a)\sigma_\alpha(\beta) \not= \sigma_\alpha(\alpha\epsilon^a)\varphi(\beta)
$$
de ce lemme $\ref{VarianteLemme3}$ est v\'erifi\'ee: c'est la remarque juste avant le lemme 3 de \cite{LW2} avec $\sigma=\sigma_\alpha$. 
On obtient alors
$$
\left| 
 \frac{\varphi(\alpha\epsilon^a)\sigma_\alpha(\beta) }
 {\sigma_\alpha(\alpha\epsilon^a)\varphi(\beta)} -1\right|\ge 
 \exp\left\{
- \Newcst{kappa:alphaepsilona}
 (\log m)\log
 \left(2+
 \frac{A+B}{\log m}\right)
 \right\}
$$
et par cons\'equent,
$$ 
\min\{A,B\}\le \Newcst{kappa:minAB}
 (\log m)\log
 \left(2+
 \frac{A+B}{\log m}\right),
$$
ce qui donne une contradiction lorsque la
 constante $\cst{kappa:LW2a}$ de ($\ref{Equation:MinorationAetB}$) est suffisamment grande. Donc nous avons $ \tau_\alpha \not= \tau_\beta$.
 
 \medskip

\indent 
{\tt Quatri\`eme \'etape.}
Montrons\footnote{La deuxi\`eme \'etape nous a permis de supposer $\tau_\alpha \not= \tau_\beta $, mais cela n'implique pas $\tau_\alpha(\epsilon) \not= \tau_\beta(\epsilon) $ car nous n'avons pas suppos\'e que $\epsilon$ \'etait un g\'en\'erateur du corps de nombres $K=\Q(\alpha)$.
} que nous avons
 $\tau_\alpha(\epsilon)\not =\pm \tau_\beta(\epsilon)$. 

\noindent
On \'ecrit l'\'equation de Siegel (\'equation ($7$) de \cite{LW2}) pour les trois plongements $\tau_\beta$, $\sigma_\alpha$ et $\tau_\alpha$. Autrement dit, on pose
$$\left\{
\begin{array}{rcl}
u_1=\tau_\beta(\alpha\epsilon^a), &
u_2=\sigma_\alpha(\alpha\epsilon^a), &
u_3=\tau_\alpha(\alpha\epsilon^a),\\ [2mm] 
v_1=\tau_\beta(\beta),&
v_2=\sigma_\alpha(\beta), &
v_3=\tau_\alpha(\beta).
\end{array} \right.
$$
La relation ($\ref{Equation:SommeSixTermes}$) est encore v\'erifi\'ee. 
On conserve ces notations pour toute la suite de la d\'emonstration.
On d\'eduit du lemme 7 de \cite{LW2} que $\tau_\beta\not\in\Sigma_\alpha(\thetanu)$; 
donc 
$$
|\tau_\beta(\epsilon)|<|\sigma_\alpha(\epsilon)|
$$ 
et par cons\'equent
$$
 \frac{|u_1|}{|u_2|}=
 \frac{|\tau_\beta(\alpha \epsilon^a)|}{|\sigma_\alpha(\alpha \epsilon^a)|}\le e^{-\Newcst{kappa:u1/u2} A}.
$$
Comme $\tau_\alpha\not=\tau_\beta$ (deuxi\`eme \'etape), le lemme 6 de \cite{LW2} livre $\tau_\alpha\not\in T_\beta(\thetanu)=\{\tau_\beta\}$; d'o\`u 
$$ 
 \frac{|v_1|}{|v_3|}=
 \frac{|\tau_\beta(\beta)|}{|\tau_\alpha(\beta)|}\le e^{-\Newcst{kappa:v1/v3} B}.
$$ 
On en d\'eduit
$$
 \frac{|u_3v_1|}{|u_2v_3|}=
 \frac{|\tau_\alpha(\alpha\epsilon^a) \tau_\beta(\beta) |}
 {|\sigma_\alpha(\alpha\epsilon^a)\tau_\alpha(\beta)|}
 \le 
 e^{-\Newcst{kappa:u3v_1/u2v3} (A+B)}.
 $$ 
 
Supposons maintenant $\tau_\alpha(\epsilon)=s_1 \tau_\beta(\epsilon)$ avec $s_1\in\{-1, 1\}$.
On a $\tau_\alpha(\alpha) u_1=s_1^a\tau_\beta(\alpha) u_3$. 
Posons 
$$
\lambda_1 = -1+s_1^a \frac{\tau_\beta(\alpha)}{\tau_\alpha(\alpha)},
$$
de sorte que $u_1-u_3=\lambda_1 u_3$. Comme $\tau_\alpha\not= \tau_\beta$, on a $u_1\not=u_3$, donc $\lambda_1\not=0$. 
 L'\'equation ($\ref{Equation:SommeSixTermes}$) devient
\begin{equation}\label{Equation:SommeSixTermesbis}
u_2v_3+\lambda_1 u_3v_2-u_1v_3-u_2v_1+u_3v_1=0. 
\end{equation}
On utilise le lemme \ref{Lemme:elementaire} avec $t=5$.
Les deux premiers termes $u_2v_3$ et $\lambda_1 u_3v_2$ sont $x_1$ et $x_2$. Il s'agit de majorer, pour chaque $i=3,4,5$, soit $|x_i| / |x_1|$, soit $|x_i| / |x_2|$, au choix. Il s'av\`ere que nous avons major\'e pr\'ec\'edemment les modules de 
$$
\frac{u_1v_3}{u_2v_3}=\frac{u_1}{u_2},
\quad
\frac{u_2v_1}{u_2v_3}=\frac{v_1}{v_3},
\quad
\frac{u_3v_1}{u_2v_3} \cdotp
$$
On d\'eduit 
$$
|u_2v_3+\lambda_1 u_3v_2|\le e^{-\Newcst{kappa:u3v2+u2v3} \min\{A,B\}}|u_2v_3|.
$$
Comme $\lambda_1\not=0$, on peut \'ecrire la majoration pr\'ec\'edente sous la forme
$$
\left|\frac{u_2v_3}{u_3v_2}+ \lambda_1 \right|
\le e^{-\Newcst{kappa:u3v2/u2v3} \min\{A,B\}}. 
$$

\indent 
La conclusion du lemme $\ref{VarianteLemme3}$ (que l'on utilise avec $\lambda=\lambda_1$) n'est pas compatible avec cette majoration; donc l'hypoth\`ese de ce lemme selon laquelle le membre de gauche est non nul n'est pas satisfaite. Autrement dit,
$$
u_2v_3+ \lambda_1 u_3 v_2=0.
$$ 
Dans l'\'equation ($\ref{Equation:SommeSixTermesbis}$), une somme 
de deux des cinq termes du membre de gauche \'etant nulle, 
la somme des trois autres termes est \'egalement nulle: 
$$
u_1v_3+u_2v_1-u_3v_1=0. 
$$
 Or on a 
$$
\left|
\frac{u_1v_3}{u_2v_1}+1\right|=
\frac{|u_3v_1|}{|u_2v_1|}=\frac{|u_3|}{|u_2|}=
\frac{|\tau_\alpha(\alpha\epsilon^a) |}
 {|\sigma_\alpha(\alpha\epsilon^a)|} 
\le e^{- \Newcst{kappa:u3/u2} A}
$$
avec
$$
\frac{u_1v_3}{u_2v_1}=
 \frac{\tau_\beta(\alpha\epsilon^a)\tau_\alpha(\beta) }
 {\sigma_\alpha(\alpha\epsilon^a)\tau_\beta(\beta)}\cdotp
 $$
Par cons\'equent, 
$$
\left|
 \frac{\tau_\beta(\alpha\epsilon^a)\tau_\alpha(\beta) }
 {\sigma_\alpha(\alpha\epsilon^a)\tau_\beta(\beta)
 }
 +1\right|\le e^{-\Newcst{kappa:taubetaalphaepsilona} \min \{A,B\}}.
 $$
Utilisant encore une fois le lemme $\ref{VarianteLemme3}$ avec $\lambda=+1$, ainsi que la formule ($\ref{Equation:MinorationAetB}$) avec une constante $\cst{kappa:LW2a}$ suffisamment grande, 
on en d\'eduit $u_1v_3+u_2v_1=0$,
d'o\`u $u_3v_1=0$, ce qui n'est pas possible. Nous avons donc 				
$\tau_\alpha(\epsilon)\not =\pm \tau_\beta(\epsilon)$. 

\medskip
\indent 
{\tt Cinqui\`eme \'etape.}
Montrons que l'on a 
\begin{equation}\label{Equation:Etape5}
\sigma_\alpha(\alpha\epsilon^a)\tau_\alpha(\beta) +\tau_\beta(\alpha\epsilon^a)\sigma_\alpha(\beta)=0.
\end{equation}

La d\'efinition de $\tau_\alpha$ implique $|\tau_\alpha(\alpha\epsilon^a)|\le |\tau_\beta(\alpha\epsilon^a)|$. La quatri\`eme 
 \'etape implique 
$\tau_\alpha(\epsilon^a)\not= \pm \tau_\beta(\epsilon^a)$. 
Comme l'unit\'e $\epsilon$ est totalement r\'eelle et que $\tau_\alpha(\epsilon)\not=\pm \tau_\beta(\epsilon)$, 
on d\'eduit
\begin{equation}\label{Equation:MajtauaMinsigmaa}
|\tau_\beta(\epsilon)|>|\tau_\alpha(\epsilon)|.
\end{equation}
De l'in\'egalit\'e ($\ref{Equation:MajtauaMinsigmaa}$) on d\'eduit 
$$
\frac{|u_3|}{|u_1|}=
\frac{|\tau_\alpha(\alpha\epsilon^a) |}
 {|\tau_\beta(\alpha\epsilon^a)|} =
\frac{|\tau_\alpha(\alpha) |}
 {|\tau_\beta(\alpha )|} \cdot
 \left(
\frac{|\tau_\alpha(\epsilon ) |}
 {|\tau_\beta( \epsilon )|} \right)^a
\le e^{- \Newcst{kappa:u3/u1b} A}.
$$ 
Consid\'erons
$$
\frac{u_2v_3}{u_1v_2}=\frac{\sigma_\alpha(\alpha\epsilon^a)\tau_\alpha(\beta) }{\tau_\beta(\alpha\epsilon^a)\sigma_\alpha(\beta)}\cdotp
$$
On utilise l'\'egalit\'e ($\ref{Equation:SommeSixTermes}$), les estimations du d\'ebut de la quatri\`eme \'etape, ainsi que le lemme \ref{Lemme:elementaire} avec $t=6$ pour d\'eduire
$$
\left|
\frac{\sigma_\alpha(\alpha\epsilon^a)\tau_\alpha(\beta) }{\tau_\beta(\alpha\epsilon^a)\sigma_\alpha(\beta)}
+1\right|\le e^{- \Newcst{kappa:sigmaalphaepsiloatau} \min \{A,B\}}.
$$ 
Gr\^ace une nouvelle fois au lemme $\ref{VarianteLemme3}$ avec $\lambda=+1$, $\varphi_1=\varphi_4=\sigma_\alpha$, $\varphi_2=\tau_\alpha$, $\varphi_3=\tau_\beta$, en utilisant 
 ($\ref{Equation:MinorationAetB}$) avec une constante $\cst{kappa:LW2a}$ suffisamment grande, on en d\'eduit 
$$
u_2v_3+u_1v_2=0,
$$
ce qui est ($\ref{Equation:Etape5}$).

\medskip
\indent 
{\tt Sixi\`eme \'etape.}
Montrons que l'on a 
\begin{equation}\label {Equation:Etape6}
|\sigma_\alpha(\epsilon)\tau_\alpha(\epsilon)| =|\tau_\beta( \epsilon)|^2.
\end{equation}
La cinqui\`eme \'etape montre qu'une somme de deux termes dans le membre de gauche de l'\'equation 
($\ref{Equation:SommeSixTermes}$) 
 est nulle, donc la somme des quatre autres termes est \'egalement nulle:
\begin{equation}\label{Equation:Siegel4}
u_2v_1+u_1v_3+u_3v_2-u_3v_1=0,
\end{equation}
c'est-\`a-dire 
$$
\sigma_\alpha(\alpha\epsilon^a)
\tau_\beta(\beta)
+
\tau_\beta(\alpha\epsilon^a)
\tau_\alpha(\beta)
+
\tau_\alpha(\alpha\epsilon^a)
\sigma_\alpha(\beta)
-
\tau_\alpha(\alpha\epsilon^a)
\tau_\beta(\beta)
=0.
$$
On suppose $|\sigma_\alpha(\epsilon)\tau_\alpha(\epsilon)|\not =|\tau_\beta( \epsilon)|^2$. 
On pose 
$$
x_1=u_2v_1=\sigma_\alpha(\alpha\epsilon^a)
\tau_\beta(\beta), \quad
x_4=-u_3v_1=-
\tau_\alpha(\alpha\epsilon^a)
\tau_\beta(\beta),
$$
et 
$$
(x_2, \; x_3)=\left\{
\begin{array}{lll}
(\tau_\alpha(\alpha\epsilon^a)
\sigma_\alpha(\beta))
,\;
\tau_\beta(\alpha\epsilon^a)
\tau_\alpha(\beta) 
&\mbox{si}&|\sigma_\alpha(\epsilon)\tau_\alpha(\epsilon)|>|\tau_\beta( \epsilon)|^2,
\\ [2mm] 
(\tau_\beta(\alpha\epsilon^a)
\tau_\alpha(\beta)
,\;
\tau_\alpha(\alpha\epsilon^a)
\sigma_\alpha(\beta))
& \mbox{si} &|\sigma_\alpha(\epsilon)\tau_\alpha(\epsilon)|<|\tau_\beta( \epsilon)|^2.
\end{array} \right.
$$
On a $x_1+x_2+x_3+x_4=0$ et 
$$
\frac{|x_4|}{|x_1|}=\frac{|u_3|}{|u_2|}=
\frac{|\tau_\alpha(\alpha\epsilon^a)|}
{|\sigma_\alpha(\alpha\epsilon^a)|}
\le e^{-\Newcst{kappa:x4/x1}A}.
$$ 

On utilise ($\ref{Equation:Etape5}$).
Si $|\sigma_\alpha(\epsilon)\tau_\alpha(\epsilon)|>|\tau_\beta( \epsilon)|^2$, on a 
$$
\frac{|x_3|}{|x_2|}=
\frac{|\tau_\beta(\alpha)|^2}{|\sigma_\alpha(\alpha)\tau_\alpha(\alpha)|}
\left|
\frac{ \tau_\beta(\epsilon)^2}{ \sigma_\alpha(\epsilon)\tau_\alpha(\epsilon) }
\right|^a
\le e^{-\Newcst{kappa:x3/x2a}A}.
$$
Si $|\sigma_\alpha(\epsilon)\tau_\alpha(\epsilon)|<|\tau_\beta( \epsilon)|^2$, on a 
$$
\frac{|x_3|}{|x_2|}=
\frac{|\sigma_\alpha(\alpha)\tau_\alpha(\alpha)|}{|\tau_\beta(\alpha)|^2}
\left|
\frac{\sigma_\alpha(\epsilon)\tau_\alpha(\epsilon) }{ \tau_\beta(\epsilon)^2}
\right|^a
\le e^{-\Newcst{kappa:x3/x2b}A}.
$$
Dans les deux cas on peut utiliser le lemme \ref{Lemme:elementaire} avec $t=4$ pour en d\'eduire
$$
\left|
\frac{x_1}{x_2}+1
\right| \le e^{-\Newcst{kappa:x1/x2+1}A}.
$$
Gr\^ace encore une fois au lemme $\ref{VarianteLemme3}$ avec $\lambda=+1$, 
$\varphi_1= \sigma_\alpha$, $\varphi_2= \tau_\beta$, 
$$
\begin{cases}
\varphi_3= \tau_\alpha, \quad \varphi_4= \sigma_\alpha 
& \hbox{si $|\sigma_\alpha(\epsilon)\tau_\alpha(\epsilon)|>|\tau_\beta( \epsilon)|^2$},
\\
\varphi_3= \tau_\beta, \quad \varphi_4= \tau_\alpha
& \hbox{si $|\sigma_\alpha(\epsilon)\tau_\alpha(\epsilon)|<|\tau_\beta( \epsilon)|^2$}
\end{cases}
$$
et 
($\ref{Equation:MinorationAetB}$) avec une constante $\cst{kappa:LW2a}$ suffisamment grande, 
on en d\'eduit $x_1+x_2=0$. Mais alors $x_3+x_4=0$. Montrons que ce n'est pas possible. D'apr\`es le lemme 6 de \cite{LW2}, on a $\sigma_\alpha\not\in T_\beta(\thetanu)$. Utilisant ($\ref{Equation:MajtauaMinsigmaa}$), on trouve
$$
\frac{|x_4|}{|x_3|}\le \max\left\{
\frac{|v_1|}{|v_2|},\; 
\frac{|u_3v_1|}{|u_1v_3|}
\right\}
\le e^{-\Newcst{kappa:x4/x3} \min\{A,B\}}<1.
$$ 
Donc $x_3+x_4\neq0$.
Ceci d\'emontre ($\ref{Equation:Etape6}$). 

\goodbreak 

\medskip
\indent 
{\tt Septi\`eme \'etape.} Fin de la d\'emonstration. 

Gr\^ace \`a l'hypoth\`ese que $\epsilon$ est une unit\'e totalement r\'eelle, l'\'equation ($\ref{Equation:Etape6}$) s'\'ecrit
$$
\sigma_\alpha(\epsilon)\tau_\alpha(\epsilon) =s_2 \tau_\beta( \epsilon)^2
$$
avec $s_2\in\{ -1, 1\}$. 
En combinant avec l'\'equation ($\ref{Equation:Etape5}$)
$$
u_2v_3=-u_1v_2,
$$
on trouve
$$
\frac{u_1v_3}{u_3v_2}=-\frac{u_1^2}{u_2u_3}=
-\frac{\tau_\beta(\alpha\epsilon^a)^2}{\sigma_\alpha(\alpha\epsilon^a)\tau_\alpha(\alpha\epsilon^a)}=
-
s_2^a
\frac{\tau_\beta(\alpha)^2}{\sigma_\alpha(\alpha)\tau_\alpha(\alpha)},
$$
d'o\`u 
$$
u_1v_3+u_3v_2=\lambda_2 u_3v_2
\quad\hbox{avec}\quad 
\lambda_2=
1-
s_2^a
\frac{\tau_\beta(\alpha)^2}{\sigma_\alpha(\alpha)\tau_\alpha(\alpha)}
\cdotp
$$
Comme $u_2\not =u_3$, l'\'equation ($\ref{Equation:Siegel4}$), qui s'\'ecrit maintenant
$$
\frac{u_2v_1}{u_3v_2}+\lambda_2=\frac{v_1}{v_2},
$$ 
entra\^{\i}ne $\lambda_2\not=0$. 
Comme 
$$
\frac{u_2v_1}{u_3v_2}=
\frac{\sigma_\alpha(\alpha\epsilon^a)\tau_\beta(\beta)}{\tau_\alpha(\alpha\epsilon^a)\sigma_\alpha(\beta)}
$$
 et que 
$$
0<\frac{|v_1|}{|v_2|}=\frac{|\tau_\beta(\beta)|}{|\sigma_\alpha(\beta)|}\le e^{- \Newcst{kapppa:v1/v2}B},
$$
on peut utiliser une derni\`ere fois le lemme $\ref{VarianteLemme3}$ avec 
$$
\lambda=\lambda_2\neq 0, 
\varphi_1=\varphi_4=\sigma_\alpha, \varphi_2=\tau_\beta, \varphi_3=\tau_\alpha
$$
 pour 
obtenir la contradiction finale avec ($\ref{Equation:MinorationAetB}$).
Ceci termine la d\'emonstration du th\'eor\`eme \ref{Theoreme:rang1} .

\section{Familles d'exemples}

Soit $\alpha$ un nombre alg\'ebrique de degr\'e $d\ge 3$ et soit $\epsilon$ une unit\'e totalement r\'eelle $\not\in \{-1,1\}$ du corps de nombres $K=\Q(\alpha)$. 
D\'esignons par $\sfA$ l'ensemble des entiers $a\in\Z$ tels que $\alpha\epsilon^a$ est de degr\'e $d$. 
Le th\'eor\`eme \ref{Theoreme:rang1} donne la majoration 
$$
\max\{\log|x|, \; \log|y|,\; |a|\} \leq \cst{kappaTheoreme} \log m
$$
pour tout triplet $(a,x,y)\in \sfA\times \Z\times \Z$ satisfaisant 
$$
 |F_a(x,y)|\leq m, \;\mbox{ avec } \; xy\neq 0.
$$
Pour $a\in\Z$, la forme binaire $F_a$ s'\'ecrit 
$$
\begin{array}{ll}
&F_a(X,Y)
\,=\,\prod_{i=1}^d \bigl(X-\sigma_i(\alpha\epsilon^a)Y\bigr)
\\ [2mm]
&=\,X^d-U_1(a) X^{d-1}Y+\cdots+(-1)^{d-1}U_{d-1}(a) XY^{d-1}+(-1)^d U_d(a)Y^d,
\end{array}
$$
o\`u $\sigma_1,\dots,\sigma_d$ d\'esignent les \'el\'ements de l'ensemble $\Phi$ des plongements de $K$ dans $\C$.
 
Les coefficients $U_1(a),\dots,U_{d}(a)$ des formes binaires $F_a$ sont donn\'es par
$$
U_h(a)=\sum_{1\le j_1<\cdots<j_h\le d} \sigma_{j_1}(\alpha\epsilon^a)\cdots\sigma_{j_h}(\alpha\epsilon^a)
\qquad (h=1,\dots,d).
$$
En particulier, 
$$
U_d(a)= \rmNorm_{K/\Q}(\alpha \epsilon^a).
$$ 
Si $\delta\in \{-1, 1\}$ d\'esigne la norme absolue de $\epsilon$ et $\nu$ le degr\'e de $K$ sur $\Q(\epsilon)$, alors on a 
$$
U_{d}(a)=\delta^{\nu} U_d(a-1)\quad\mbox{avec} \quad U_d(0)= \rmNorm_{K/\Q}(\alpha).
$$

Pour $h=1,\dots,d-1$, la suite $\bigl(U_h(a)\bigr)_{a\in\Z}$ 
v\'erifie une relation de r\'ecurrence lin\'eaire faisant intervenir le polyn\^ome irr\'eductible de $\epsilon$. 
Nous allons pr\'eciser ces r\'ecurrences quand l'unit\'e $\epsilon$ est quadratique (\S $\ref{SS:ExtensionsCorpsQuadratiques}$ et \S $\ref{SS:ExempleExtensionsCorpsQuadratiques}$) et quand $\epsilon$ est cubique (\S $\ref{SS:ExtensionsCorpsCubiques}$ et \S $\ref{SS:CubiquesSimplesShanks}$). 

\subsection{Extensions de corps quadratiques r\'eels}\label{SS:ExtensionsCorpsQuadratiques}

Nous supposons ici que le corps $K$ contient un corps quadratique r\'eel $k$ et que $\epsilon$ est une unit\'e de $k$. Notons $\epsilonbar$ le conjugu\'e galoisien de $\epsilon$. La norme absolue de $\epsilon$ est $\delta= \epsilon\epsilonbar\in\{-1, 1\}$. 

Soit $h\in\{1,\dots,d-1\}$. Comme $\sigma_{j_1}(\epsilon)\cdots\sigma_{j_h}(\epsilon)$ peut s'\'ecrire $\epsilon^\ell{\epsilonbar\, }^{h-\ell}$ avec $0\le \ell\le h$, la suite $\bigl(U_h(a)\bigr)_{a\in\Z}$ est une combinaison lin\'eaire des $h+1$ suites 
$$
\bigl((\epsilon^\ell{\epsilonbar\, }^{h-\ell})^a\bigr)_{a\in\Z}, \;\ell=0,\dots,h.
$$
 Donc $\bigl(U_h(a)\bigr)_{a\in\Z}$ est une suite r\'ecurrente lin\'eaire d'ordre $h+1$ dont le polyn\^ome caract\'eristique est
$$
\prod_{\ell=0}^h (T-\epsilon^\ell{\epsilonbar\, }^{h-\ell}).
$$
En \'ecrivant 
$$
F_a(X,Y)
=
 \delta^{ad/2} \rmN_{K/\Q}(\alpha) \prod_{i=1}^d \bigl(Y-\sigma_i(\alpha^{-1}\epsilon^{-a})X\bigr),
$$
on voit que $\bigl(U_h(a)\bigr)_{a\in\Z}$ est aussi une suite r\'ecurrente lin\'eaire d'ordre $d-h+1$, dont le polyn\^ome caract\'eristique est 
$$
\prod_{\ell=0}^{d-h} (T-\delta^{d/2}\epsilon^{-\ell}{\epsilonbar\, }^{-d+h+\ell}).
$$

Pour $h=1$, la suite $\bigl(U_1(a)\bigr)_{a\in\Z}$ est r\'ecurrente lin\'eaire d'ordre $2$, son polyn\^ome caract\'eristique \'etant le polyn\^ome irr\'eductible de $\epsilon$;
noter que 
$$
U_1(a)=\trace_{K/\Q}(\alpha\epsilon^a).
$$
 De m\^eme, 
pour $h=d-1$, la suite $\bigl(U_{d-1}(a)\bigr)_{a\in\Z}$ est r\'ecurrente lin\'eaire d'ordre $2$, son polyn\^ome caract\'eristique \'etant le polyn\^ome irr\'eductible de $\delta^{d/2}\epsilon^{-1}$; noter que 
$$
U_{d-1}(a)=U_d(a)\trace_{K/\Q}\bigl(\alpha^{-1}\epsilon^{-a}\bigr).
$$
Pour $h=2$, le polyn\^ome caract\'eristique de la r\'ecurrence lin\'eaire d'ordre $3$ v\'erifi\'ee par la suite $\bigl(U_2(a)\bigr)_{a\in\Z}$ est 
$$ 
(T-\delta)(T-\epsilon^2)(T-{\epsilonbar\, }^2).
$$
Les suites $\bigl(U_2(a+2)\bigr)_{a\in\Z}$, $\bigl(U_2(a+1)\bigr)_{a\in\Z}$, $\bigl(U_2(a)\bigr)_{a\in\Z}$ et $\bigl(\delta^a\bigr)_{a\in\Z}$ sont lin\'eai\-rement d\'ependantes, donc $\bigl(U_2(a)\bigr)_{a\in\Z}$ v\'erifie aussi une relation de r\'ecurrence de la forme
$$
U_2(a+2)=c_1U_2(a+1)+c_2U_2(a)+c_3\delta^a \qquad (a\in\Z),
$$ 
o\`u $c_1$ et $c_2$ sont d\'etermin\'es par
$$
T^2-c_1T-c_2=(T-\epsilon^2)(T-{\epsilonbar\, }^2),
$$
tandis que $c_3$ est d\'etermin\'e par les conditions initiales $U_2(-1),U_2(0),U_2(1)$. 

\subsection{Exemples d'extensions de corps quadratiques r\'eels}\label{SS:ExempleExtensionsCorpsQuadratiques}
Nous explicitons les formules du \S $\ref{SS:ExtensionsCorpsQuadratiques}$
dans le cas particulier de corps de nombres consid\'er\'es par L. Bernstein et H.~Hasse (voir \cite{B--H}), \`a savoir 
les corps $K=\Q(\omega)$ o\`u $\omega^m = D^m \pm d$ avec $D\in \N, d\in \Z, d|D$, pour lesquels ils ont exhib\'e les unit\'es 
$$
\frac{1}{d}
 (\omega^t-D^t)^{m/t}\quad \mbox{o\`u} \quad t|m \;\mbox{avec}\; 1\leq t<m.
$$ 
 
Ici nous allons consid\'erer la famille de corps $K$ de degr\'e $2n$ d\'efinie par 
$$
K=\Q(\omega) \; \mbox{ o\`u } \; \omega = \sqrt[2n]{D^{2n}+c} > 1\;\mbox{ avec }\; D\in \N, \; n\ge 2,\; c \in \{-1, 1\}
$$ 
 et utiliser les deux unit\'es ind\'ependantes de $K$ donn\'ees par 
 $$
\alpha =D+ \omega \quad \mbox{et} \quad \epsilon=D^n+ \omega^n,
$$
 l'unit\'e $\alpha$ \'etant de degr\'e $2n$ et l'unit\'e $\epsilon$ quadratique.

\begin{proposition}\label{PropositionExtensionQuadratique}
Soit
$$
\begin{array}{lll}
F_a(X,Y) 
&=&X^{2n}-U_1(a)X^{2n-1}Y+U_2(a)X^{2n-2}Y^2-\cdots \\ [1mm]
&&\hskip 3.72cm - U_{2n-1}(a)XY^{n-1}+U_{2n}(a) Y^{2n}
\end{array}
$$
la version homog\'en\'eis\'ee du polyn\^ome minimal $F_a(X,1)$ de $\alpha \epsilon^a$.
Cette forme binaire s'\'ecrit
$$
F_a(X,Y)\,=\,
\bigl(
(X-\epsilon^aD)^n-\epsilon^{na}\omega^nY^n
\bigr)
\bigl(
(X- {\epsilonbar\ }^aD)^n+{\epsilonbar\ }^{na}\omega^n Y^n
\bigr).
$$
De plus, les propri\'et\'es suivantes sont v\'erifi\'ees.

\mbox{\rm(i)} On a l'\'egalit\'e $U_{2n}(a)=(-c)^{na+1}$. 

\mbox{\rm (ii)}
Pour $1\le h\le 2n$, la suite $\bigl(U_h(a)\bigr)_{a\in\Z}$ v\'erifie une relation de r\'ecurrence lin\'eaire d'ordre $\min\{h+1,2n-h+1\}$. 

\mbox{\rm (iii)} Pour $h=1$ on a 
$$
U_1(a)=nD(\epsilon^a+{\epsilonbar\, }^a).
$$ 
On en d\'eduit
$$
U_1(a+2)=2D^n U_1(a+1)+cU_1(a),
$$
avec les conditions initiales 
$$
U_1(0)=2nD, \quad 
U_1(1)=2nD^{n+1}.
$$

\mbox{\rm(iv)}
Pour $h=2n-1$ on a 
$$
 U_{2n-1}(a)=
 (-c)^{(n-1)a} n D^{n-1}
 \bigl(
 D^n(\epsilon^a+{\epsilonbar\ }^a)+
 (-1)^{n-1} \omega^n 
 (\epsilon^a-{\epsilonbar\ }^a)
\bigr). 
$$ 
On en d\' eduit
$$
U_{2n-1}(a+2)= 2(-c)^{n-1}D^n U_{2n-1}(a+1)+cU_{2n-1}(a)
$$
avec les conditions initiales 
$$
U_{2n-1}(0)=2nD^{2n-1} , \quad 
U_{2n-1}(1)=
\begin{cases}
2nD^{n-1} &\hbox{ si $n$ est pair,}
\\
2n D^{n-1} (2D^{2n}+c)
&\hbox{ si $n$ est impair.}
\end{cases}
$$

\mbox{\rm (v)}
Pour $h=2$ on a
$$
U_2(a)=
\begin{cases}
4D^2(-c)^a+\epsilonbar\epsilon^{2a}+\epsilon{\epsilonbar\, }^{2a}
\hbox{ pour $n=2$,}
\\
\displaystyle
n^2D^2(-c)^a+\frac{n(n-1)}{2} D^2(\epsilon^{2a}+{\epsilonbar\, }^{2a})
\hbox{ pour $n\ge 3$.}
\end{cases}
$$
La suite $\bigl(U_2(a)\bigr)_{a\in\Z}$ v\'erifie donc la relation de r\'ecurrence 
lin\'eaire 
$$
U_2(a+3)=(4D^{2n}+c)U_2(a+2) +(4cD^{2n}+1) U_2(a+1) -cU_2(a)
$$
avec les conditions initiales
\begin{align}
\notag
&
\quad
U_2(0)=n(2n-1)D^2
\\
\notag
&
{\begin{cases}
U_2(-1)=2D^2(4D^4+c), \quad U_2(1)=-6cD^6 
&\hbox{pour $n=2$,}
\\
U_2(-1)=U_2(1)=2n(n-1) D^{2n+2} -cnD^2 
&\hbox{pour $n\ge 3$}.
\end{cases}
}\end{align} 
On en d\'eduit la relation de r\'ecurrence lin\'eaire non homog\`ene 
$$
U_2(a+2)=2(2D^{2n}+c)U_2(a+1) - U_2(a)+ 4cn^2D^2(D^{2n}+c) (-c)^{a}.
$$

\end{proposition}

\begin{proof}[D\'emonstration] 
Le polyn\^ome irr\'eductible de $\alpha$ est
$$
F_0(X,1)
\,=\,
(X-D)^{2n}-D^{2n}-c
\,= \,\sum_{h=0}^{2n-1} (-1)^h \binom{2n}{h}D^h X^{2n-h}-c,
$$
ce qui donne 
$$
U_h(0)= \binom{2n}{h}D^h \quad(1\le h\le 2n-1).
$$
En particulier,
$$
U_1(0)=2nD, \quad
U_2(0)=n(2n-1)D^2,
\quad
U_{2n-1}(0)=2nD^{2n-1}.
$$
Le polyn\^ome 
$$
(X-\epsilon^aD)^n-\epsilon^{na}\omega^n,
$$
de degr\'e $n$, est \`a coefficients dans l'anneau $\Z[\epsilon]$ et s'annule au point $\alpha \epsilon^a$. Le conjugu\'e galoisien de $\omega^n$ est $-\omega^n$, celui de $\epsilon$ est $\epsilonbar=D^n-\omega^n$. Le produit 
$$
\big(
(X-\epsilon^aD)^n-\epsilon^{na}\omega^n
\bigr)
\big(
(X- {\epsilonbar\ }^aD)^n+ {\epsilonbar\ }^{na}\omega^n\bigr)
$$
est \`a coefficients dans $\Z$ et s'annule au point $\alpha \epsilon^a$: c'est le polyn\^ome $F_a(X,1)$. 

D\'efinissons les suites $\bigl( V_h(a)\bigr)_{a\in\Z}$ et  $\bigl( W_h(a)\bigr)_{a\in\Z}$ pour $0\le h\le 2n$ par 
$$ 
(X-\epsilon^aD)^n (X- {\epsilonbar\ }^aD)^n=\sum_{h=0}^{2n} (-1)^h V_h(a) X^{2n-h}
$$
et
$$
(X-\epsilon^aD)^n{\epsilonbar\ }^{na}
-(X- {\epsilonbar\ }^aD)^n \epsilon^{na}
=\sum_{h=0}^{2n} (-1)^h W_h(a) X^{2n-h}.
$$
On a 
$$
U_0(a)=V_0(a)=1, 
\quad W_h(a)=0  \quad (0\le h\le n-1), 
$$
$$
V_{2n}(a)=(-c)^{na} D^{2n}, \quad W_{2n}(a)=0,
$$
$$
U_h(a)=V_h(a)+W_h(a)\omega^n \quad (0\le h\le 2n-1),
$$
$$
V_h(a)=D^h \sum_{
\atop{0\le i,j\le n}{ i+j=h}}
\binom{n}{i} \binom{n}{j} \epsilon^{ai} {\epsilonbar\ }^{aj}
 \quad (1\le h\le 2n)
$$
et, pour $n\le h\le 2n-1$,
$$
W_h(a)= (-1)^{n-1}(-c)^{(h-n)a}\binom{n}{2n-h} D^{h-n}
\bigl( \epsilon^{a(2n-h)}-{\epsilonbar\ }^{a(2n-h)}\bigr).
$$

\mbox{\rm (i)} 
La norme absolue de $\epsilon$ est $\delta=\epsilon\epsilonbar=-c$, sa norme de $K$ sur $\Q$ est $\delta^n=(-c)^n\in\{-1, 1\}$. La norme de $\alpha$ est $-c$, celle de $\alpha\epsilon^a$ est donc $U_{2n}(a)=(-c)^{na+1}$.

\mbox{\rm (ii)} La remarque au d\'ebut du $\S \ref{SS:ExtensionsCorpsQuadratiques}$ montre que la suite $\bigl( U_h(a)\bigr)_{a\in\Z}$ v\'erifie la relation de r\'ecurrence lin\'eaire d'ordre $h+1$ de polyn\^ome caract\'eristique 
$$
\prod_{\ell=0}^h (T-\epsilon^\ell {\epsilonbar\ }^{h-\ell})
$$
ainsi que la relation de r\'ecurrence lin\'eaire d'ordre $2n-h+1$ de polyn\^ome ca\-rac\-t\'eristique 
$$
\prod_{k=0}^{2n-h} 
\bigl(T- (-c)^{n-h} \epsilon^k {\epsilonbar\ }^{2n-h-k}\bigr).
$$

\mbox{\rm (iii)} 
Pour $1\le h\le n-1$, on a  $U_h(a)=V_h(a)$, ce qui donne
$$
U_h(a)=D^h \sum_{
\atop{0\le i,j\le n}{ i+j=h}}
\binom{n}{i} \binom{n}{j} \epsilon^{ai} {\epsilonbar\ }^{aj}.
$$
On en d\'eduit  $U_1(a)=nD(\epsilon^a+{\epsilonbar\, }^a)$.
La relation de r\'ecurrence lin\'eaire pour la suite $\bigl(U_1(a)\bigr)_{a\in\Z}$ r\'esulte de 
$$
(T-\epsilon)(T-\epsilonbar)=T^2-2D^nT-c.
$$

\mbox{\rm (iv)}  
On a 
$$
\left\{
\begin{matrix}
 V_{2n-1}(a)&=&(-c)^{(n-1)a}nD^{2n-1} (\epsilon^a + {\epsilonbar\ }^a),
 \hfill
 \\
W_{2n-1}(a)&=&(-1)^{n-1} (-c)^{(n-1)a} n D^{n-1}(\epsilon^a - {\epsilonbar\ }^a), \hfill
 \\
U_{2n-1}(a)&=&V_{2n-1}(a)+W_{2n-1}(a)\omega^n. \hfill
\end{matrix}
\right.
$$
Pour $a=1$ on trouve
$$
 V_{2n-1}(1)=2(-c)^{n-1}n D^{3n-1},\quad
 W_{2n-1}(1)=2c^{n-1}  n D^{n-1}\omega^n.
$$
La relation de r\'ecurrence lin\'eaire pour $\bigl(U_{2n-1}(a)\bigr)_{a\in\Z}$ r\'esulte de 
$$
\bigl(T- (-c)^{n-1}\epsilon\bigr)\bigl(T- (-c)^{n-1}\epsilonbar\bigr)=T^2-2(-c)^{n-1}D^nT-c.
$$ 

\mbox{\rm (v)} 
Comme 
$$
\epsilon^2=2D^{2n}+c+2D^n\omega^n,
$$
la trace de $\epsilon^2$ est $2(2D^{2n}+c)$, sa norme est $1$, et son polyn\^ome irr\'eductible est
$$
T^2-2(2D^{2n}+c)T+1.
$$
Le polyn\^ome caract\'eristique de la relation de r\'ecurrence homog\`ene satisfaite par la suite $\bigl(U_2(a)\bigr)_{a\in\Z}$ est 
$$
(T-\epsilon^2)(T-{\epsilonbar\, }^2)(T+c)= 
T^3-(4D^{2n}+c)T^2-(4cD^{2n}+1)T+c.
$$
Lorsque $n\ge 3$, 
$$
U_2(a)=V_2(a)=(-c)^an^2D^2+\frac{n(n-1)}{2} D^2(\epsilon^{2a}+ {\epsilonbar\ }^{2a}).
$$ 

Il reste \`a consid\'erer le cas particulier $n=2$. 
L'unique sous-corps quadratique de $K$ est $\Q(\omega^2)$. Pour tout $a\in\Z$ on a $\Q(\alpha\epsilon^a)=K$; donc ici $\sfA=\Z$. 
Les conjugu\'es de $\alpha\epsilon^a$ sont 
$$
\epsilon^a (\omega+D),\quad
{\epsilonbar\, }^a (i\omega+D),\quad
\epsilon^a (-\omega+D),\quad
{\epsilonbar\, }^a(-i\omega+D);
$$
donc 
$$
U_1(a)=2D(\epsilon^a+{\epsilonbar\, }^a),
\quad 
U_2(a)= 
(-c)^a 4 D^2 +\epsilonbar \epsilon^{2a}+\epsilon {\epsilonbar\, }^{2a}.
$$
 Les conjugu\'es de $\alpha^{-1}\epsilon^{-a}$ sont 
$$
c \epsilon^{-a+1} (\omega-D),\quad
c\, {\epsilonbar\, }^{-a+1} (i\omega-D),\quad
c \epsilon^{-a+1} (-\omega-D),\quad
c\, {\epsilonbar\, }^{-a+1} (-i\omega-D);
$$
donc 
$$
U_3(a)=-c\, \trace_{K/\Q}(\alpha^{-1}\epsilon^{-a})=2D(\epsilon^{-a+1}+{\epsilonbar\, }^{-a+1}).
$$ 
En conclusion, nous avons 
$$
F_a(X,Y)=X^4-U_1(a)X^3Y+U_2(a)X^2Y^2-U_3(a)XY^3-cY^4, 
$$
alors que 
$$\left\{
\begin{array}{lll}
U_1(a+2)&=&2D^2U_1(a+1)+cU_1(a)\\
&& \mbox{avec }\;
U_1(0)=4D, \; 
 U_1(1)=4D^3, \\ [3mm]
U_2(a+2)&=&(4D^4+2c)U_2(a+1)-U_2(a)-(-c)^{a+1}16D^2(D^4+c)\\
&& \mbox{avec }\; U_2(0)=6D^2, \;
U_2(1)=-6cD^2, \\ [3mm]
U_2(a+3)&=&(4D^4+c)U_2(a+2)+(4cD^4+1)U_2(a+1)-cU_2(a)\\
&& \mbox{avec }\; U_2(0)=6D^2, \;
U_2(1)=-6cD^2, \;U_2(2)= -8cD^6-2D^2, \\ [3mm]
U_3(a+2)&=& -2cD^2U_3(a+1)+cU_3(a)\\ 
 & & \mbox{avec }\;
U_3(0)=4D^3, \;
U_3(1)=4D.
\end{array}\right.
 $$
\end{proof}

\subsection{Extensions de corps cubiques totalement r\'eels}\label{SS:ExtensionsCorpsCubiques}
Nous \'etudions maintenant le cas particulier du th\'eor\`eme \ref{Theoreme:rang1}
o\`u $\epsilon$ est une unit\'e cubique totalement r\'eelle. 

On suppose donc que le corps de nombres $K=\Q(\alpha)$ contient un corps cubique totalement r\'eel $k$. Soit $\epsilon$ une unit\'e de $K$ diff\'erente de $1, -1$; notons $\epsilon_1,\epsilon_2,\epsilon_3$ les conjugu\'es galoisiens de $\epsilon$ avec $\epsilon_1=\epsilon$. La norme absolue de $\epsilon$ est 
$$
\delta=\rmNorm_{k/\Q}(\epsilon)= \epsilon_1\epsilon_2\epsilon_3\in\{-1,1 \}.
$$ 
Soit $h\in\{1,\dots,d-1\}$. Comme $\sigma_{j_1}(\epsilon)\cdots\sigma_{j_h}(\epsilon)$ peut s'\'ecrire 
$$
\epsilon_1^{\ell_1} \epsilon_2^{\ell_2} \epsilon_3^{\ell_3}\quad \mbox{avec} \quad \ell_1+\ell_2+\ell_3= h,
$$
la suite $\bigl(U_h(a)\bigr)_{a\in\Z}$ est une combinaison lin\'eaire des $\binom{h+2}{2}$ suites 
$$
\bigl((\epsilon_1^{\ell_1}\epsilon_2^{\ell_2}\epsilon_3^{\ell_3})^a\bigr)_{a\in\Z},\quad \ell_1+\ell_2+\ell_3=h.
$$ 
Donc $\bigl(U_h(a)\bigr)_{a\in\Z}$ est une suite r\'ecurrente lin\'eaire d'ordre $\frac{(h+1)(h+2)}{2}$, dont le polyn\^ome caract\'eristique est 
$$
\prod_{\ell_1+\ell_2+\ell_3=h} \bigl(T-\epsilon_1^{\ell_1}\epsilon_2^{\ell_2}\epsilon_3^{\ell_3}\bigr).
$$
Notons 
$$
T^3-rT^2+sX-\delta
$$
 le polyn\^ome irr\'eductible de $\epsilon$ sur $\Q$. Le fait que ce polyn\^ome soit irr\'eductible s'\'ecrit $r-s\neq 1-\delta$ et $r+s\neq -1-\delta$.
 
Pour $a\in\Z$ on a 
$$\left\{
\begin{array}{rcl}
U_1(a)&=& \alpha_1 \epsilon_1^a+\alpha_2 \epsilon_2^a+ \alpha_3 \epsilon_3^a,\\ [2mm]
U_{1}(a+3)&=&rU_{1}(a+2)-sU_{1}(a+1)+\delta U_{1}(a).
\end{array} \right. 
$$
Notons que $\delta^{d/3} =\delta^d$. 
Comme $U_d(a)= \delta^{ad} \rmN_{K/\Q}(\alpha)$ et comme le polyn\^ome irr\'eductible de $\delta^d \epsilon^{-1} $ est 
$$
T^3-\delta^{d+1} sT^2+\delta rT-\delta^{d+1},
$$
on a aussi, pour $a\in\Z$,
$$\left\{
\begin{array}{rcl}
U_{d-1}(a)&= &U_d(a) \bigl(
 \alpha_1^{-1} \epsilon_1^{-a}+ \alpha_2^{-1} \epsilon_2^{-a}+\alpha_3^{-1} \epsilon_3^{-a}
 \bigr),
\\ [2mm]
U_{d-1}(a+3)&=& \delta^{d+1} sU_{d-1}(a+2)-\delta rU_{d-1}(a+1)+ \delta^{d+1} U_{d-1}(a).
\end{array} \right.
$$

Consid\'erons le cas particulier o\`u $\alpha$ est de degr\'e $3$. Noter que dans ce cas, l'ensemble $\Z\setminus \sfA$ contient au plus un \'el\'ement. On a 
$$ 
F_a(X,Y) = X^3-U_1(a) X^2Y+U_{2}(a) XY^2-U_3(a)Y^3
$$
avec $U_3(a)=\delta^a \rmN_{K/\Q}(\alpha)$. 
\'Ecrivons les conditions initiales pour la suite r\'ecurrente lin\'eaire $\bigl(U_1(a)\bigr)_{a\in\Z}$. 
Des relations 
$$
\epsilon_1+\epsilon_2+\epsilon_3=r,
\quad
\epsilon_1\epsilon_2+\epsilon_1\epsilon_3+\epsilon_2\epsilon_3= s,
$$
nous d\'eduisons
$$
\epsilon_1^2+\epsilon_2^2+\epsilon_3^2=r^2-2 s
$$
et
$$
\epsilon_1^3+\epsilon_2^3+\epsilon_3^3=r^3-3 rs+3\delta.
$$
\'Ecrivons 
$$
\alpha=A+B\epsilon+C\epsilon^2.
$$
Alors 
$$\left\{
\begin{array}{rcl}
U_1(-1)&=&A\delta s+3B+Cr,\\ [2mm]
U_1(0)&=&
3A+Br+C(r^2-2s),
\\ [2mm]
U_1(1)&=&
Ar+B(r^2-2s)+C(r^3 -3rs +3\delta).
\end{array} \right.
$$
Dans le cas particulier $\alpha=\epsilon$, c'est-\`a-dire $A=0$, $B=1$, $C=0$, les calculs se simplifient et les conditions initiales deviennent
$$
U_1(-1)=3,
\quad
U_1(0)=r,
\quad
U_1(1)=r^2-2s
$$
et de m\^eme
$$
U_2(-1)=3,
\quad
U_2(0)=s,
\quad
U_2(1)=s^2-2\delta r.
$$

\subsection{Exemples de corps cycliques cubiques}\label{SS:CubiquesSimplesShanks}
Consid\'erons la famille des corps cubiques cycliques des plus simples de Shanks \cite{LW3}. 
Soit $n\in\Z$. D\'esignons par $\lambda$ (d\'ependant de $n$) une racine du polyn\^ome
$$ 
f(X)=X^3-(n-1)X^2 -(n+2)X -1.
$$
Les racines de $f$ sont 
$$
\lambda_1=\lambda, \quad \lambda_2=-\frac{1}{\lambda+1},\quad \lambda_3=-\frac{\lambda+1}{\lambda}\cdot
$$
Le groupe des unit\'es de $K$ (modulo $\{-1, 1\}$) est engendr\'e par deux unit\'es quelconques de $\{\lambda_1, \lambda_2, \lambda_3\}$. 

Soient $b_1,b_2,c_1,c_2$ des entiers rationnels avec $b_1c_2\not=b_2c_1$. En prenant $\epsilon=\lambda_1^{b_1}\lambda_2^{b_2}$ et 
$\alpha=\lambda_1^{c_1}\lambda_2^{c_2}$, nous sommes dans la situation de la section 
$\ref{SS:ExtensionsCorpsCubiques}$ et nous pouvons appliquer le corollaire $\ref{Corollaire:cubiquetotalementreel}$ au polyn\^ome minimal de 
$$
\alpha \epsilon^a=\lambda_1^{ab_1+c_1}\lambda_2^{ab_2+c_2} \qquad (a\in\Z).
$$
Pour donner un exemple explicite, prenons $b_1=c_2=0$, $b_2=c_1=1$. 
La famille infinie $F_a(X,Y)$ de formes binaires est donc obtenue en tordant la version homog\'en\'eis\'ee $F(X,Y)$ de $f(X)$ par les puissances ${\lambda}_2^a$ de $\lambda_2$, ($a\in\Z$). 
 
Le polyn\^ome minimal $F_a(X,1)$ de $\lambda_1 \lambda_2^a$ fait intervenir
 les suites r\'ecurrentes lin\'eaires $(s_a)_{a\ge 0}$ et $(t_a)_{a\ge 0}$ d'ordre $3$ d\'efinies par
$$
s_a= \lambda _1\lambda_2^a+\lambda_2 \lambda_3^a+ \lambda_3 \lambda_1^a,
\qquad
t_a= \lambda _1^{-1}\lambda_2^{-a}+\lambda_2^{-1} \lambda_3^{-a}+ \lambda_3^{-1} \lambda_1^{-a}.
$$
On a 
$$
\begin{array}{lll}
F_a(X,Y)&=&\left( X- \lambda_1\lambda_2^aY \right)\left( X- \lambda_2\lambda_3^aY \right)\left( X- \lambda_3\lambda_1^aY \right)\\ [2mm]
&=&X^3-s_aX^2Y+t_aXY^2-Y^3
\end{array}
$$
avec
$$
s_a= \trace_{K/ \Q}(\lambda_1 \lambda_2^a),
\quad 
t_a = \trace_{K/ \Q}(\lambda_1^{-1} \lambda_2^{-a}) = \trace_{K/ \Q}(\lambda_1 \lambda_3^{-a+1}).
$$
En particulier, 
$$
s_2=\trace_{K/ \Q}(\lambda_3 \lambda_1^{2})=\trace_{K/ \Q}(-(\lambda_1+1) \lambda_1) =\trace_{K/ \Q}(-\lambda_1-\lambda_1^2)=-n^2 -n-4
$$
et
$$
t_2=\trace_{K/ \Q}(\lambda_3 \lambda_2^{-1})=\trace_{K/ \Q}((\lambda_1 +1)^2 /\lambda_1)=\trace_{K/ \Q}(\lambda_1+2+\lambda_1^{-1} )=3.
$$
Ainsi
$$\left\{
\begin{array}{l}
s_0= n-1,\\ [1mm]
s_1= -n-2,\\ [1mm]
s_2= -n^2-n-4,\\ [1mm]
s_{a+3}=\!(n\!-\!1)s_{a+2}\!+\!(n\!+\!2)s_{a+1}\!+\!s_a,\\
\end{array}\right.
\!\! \!\!
\left\{
\begin{array}{l}
t_0= -n-2,\\ [1mm] 
t_1= n-1,\\ [1mm]
t_2=3\\ [1mm]
t_{a+3}=\!-(n\!+\!2)t_{a+2}\!-\!(n\!-\!1)t_{a+1}\!+\!t_a,\\
\end{array} \right. 
$$
pour $a\in\Z$.

Le corollaire $\ref{Corollaire:cubiquetotalementreel}$ affirme 
qu'il existe une constante effectivement calculable $\cst{kappacorollaire}>0$, ne d\'ependant que de $n$ (donc de $\lambda$), ayant la propri\'et\'e suivante:
Pour tout $m\geq 2$, tout triplet $(a,x,y)\in \Z^3$ satisfaisant 
$$
 |F_a(x,y)|\leq m, \;\mbox{ avec } \; xy\neq 0,
 $$
v\'erifie
$$
\max\{\log|x|, \; \log|y|,\; |a|\} \leq \cst{kappacorollaire} \log m.
$$

\medskip
 \medskip\noindent
 {\bf Remerciements.} 
 Le premier auteur a b\'en\'efici\'e d'un soutien financier du CRSNG.
 Le second auteur a b\'en\'efici\'e d'un s\'ejour \`a l'Universit\'e Roma Tre pendant lequel il a travaill\'e sur ce texte; il remercie Francesco Pappalardi pour son invitation, ainsi que Corrado Falcolini pour son aide avec le logiciel de calcul {\sl Mathematica}. 

\bibliographystyle{amsplain}

\end{document}